\documentclass[twoside]{aiml22}

\usepackage{aiml22macro}

\usepackage{amsmath}
\usepackage{amssymb}
\usepackage{amsrefs}
\usepackage{thinsp}
\usepackage{tikz}
\usepackage{proof}

\newcommand{\tuple}[1]{\ensuremath{\langle{#1}\rangle}}
\newcommand{\eq}{\approx}
\newcommand{\N}{\ensuremath{\mathbb{N}}}

\newcommand{\ops}{\ensuremath{\mathop{\star}}}
\newcommand{\lang}{\mathcal{L}}
\newcommand{\langs}{\mathcal{L}_s}
\newcommand{\model}[1]{{\mathfrak{{#1}}}}
\newcommand{\semvalue}[1]{\ensuremath{\left\|{#1}\right\|}}
\newcommand{\All}[1]{\ensuremath{(\all{#1})}}
\newcommand{\Exi}[1]{\ensuremath{(\exi{#1})}}
\newcommand{\all}{\ensuremath{\forall}}
\newcommand{\exi}{\ensuremath{\exists}}
\newcommand{\K}{\cls{K}}
\newcommand{\V}{\cls{V}}
\newcommand{\mK}{\mathit{m}\K}
\newcommand{\mV}{\mathit{m}\V}
\newcommand{\f}{\ensuremath{\varphi}}
\newcommand{\p}{\ensuremath{\psi}}
\newcommand{\x}{\ensuremath{\chi}}

\newcommand{\pd}{\cdot}
\newcommand{\zr}{{\rm f}}
\newcommand{\ut}{{\rm e}}
\newcommand{\mt}{\land}
\newcommand{\jn}{\lor}
\newcommand{\FLe}{\ensuremath{\mathrm{FL}_\mathrm{e}}}

\renewcommand{\a}{\ensuremath{\alpha}}
\renewcommand{\b}{\ensuremath{\beta}}

\newcommand{\bo}{\ensuremath{\Box}}
\newcommand{\di}{\ensuremath{\Diamond}}

\newcommand{\De}{\mathrm{\Delta}}
\newcommand{\Ga}{\mathrm{\Gamma}}

\newcommand{\Si}{\mathrm{\Sigma}}
\newcommand{\fosc}[1]{\vDash^{\forall}_{#1}}
\newcommand{\der}[1]{\vdash_{_\lgc{#1}}}
\newcommand{\lgc}[1]{\mathsf{#1}}
\renewcommand{\alg}[1]{\mathbf{#1}}
\newcommand{\cls}[1]{\mathcal{#1}}
\newcommand{\mcls}[1]{\mathit{m}\cls{#1}}

\newcommand*{\pfa}[1]{\mbox{\footnotesize $#1$}} 	 

\newcommand{\mfml}{\ensuremath{{\rm Fm}_{\bo}}}
\newcommand{\ofml}{\ensuremath{{\rm Fm}_{\forall}^{1}}}
\newcommand{\ofmls}{\ensuremath{{\rm Fm}_{\forall}^{1+}}}

\newcommand*{\md}{{\rm md}}
\newcommand*{\height}{{\rm ht}}
\newcommand*{\fFLe}{\lgc{\all^+_1 FL_e}}
\newcommand*{\fFLew}{\lgc{\all^+_1 FL_{ew}}}
\newcommand*{\fFLec}{\lgc{\all^+_1 FL_{ec}}}
\newcommand*{\seq}{{\vphantom{A}\Rightarrow{\vphantom{A}}}}
\newcommand*{\rseq}{\Rightarrow}
\newcommand{\idr}{(\textsc{id})}

\newcommand{\flr}{(\zr\!\rseq)}
\newcommand{\frr}{(\rseq\!\zr)}
\newcommand{\tlr}{(\ut\!\rseq)}
\newcommand{\trr}{(\rseq\!\ut)}
\newcommand{\olr}{(\jn\!\rseq)}
\newcommand{\orr}{(\rseq\!\jn)}
\newcommand{\alr}{({\mt\!\rseq})}
\newcommand{\arr}{({\rseq\!\mt})}
\newcommand{\falr}{({\all\!\rseq})}
\newcommand{\farr}{(\rseq\!\all)}
\newcommand{\elr}{(\exi\!\rseq)}
\newcommand{\err}{(\rseq\!\exi)}
\newcommand{\pdlr}{(\pd\!\rseq)}
\newcommand{\pdrr}{(\rseq\!\pd)}
\newcommand{\ilr}{(\to\rseq)}
\newcommand{\irr}{(\rseq\to)}

\newcommand{\wkr}{\textup{\sc (w)}}
\newcommand{\cnr}{\textup{\sc (c)}}

\def\lastname{Cintula, Metcalfe and Tokuda}

\begin{document}

\begin{frontmatter}
  \title{Algebraic Semantics for One-Variable Lattice-Valued Logics}
     \author{Petr Cintula}\footnote{Supported by Czech Science Foundation project GA22-01137S and by RVO 67985807.}
    \address{Institute of Computer Science of the Czech Academy of Sciences\\ Prague, Czech Republic \\ \texttt{cintula@cs.cas.cz}}
   \author{George Metcalfe}\footnote{Supported by Swiss National Science Foundation grant 200021\textunderscore 184693.}
   \author{Naomi Tokuda}
  \address{Mathematical Institute, University of Bern, Switzerland \\ $\{$\texttt{george.metcalfe,naomi.tokuda}$\}$\texttt{@unibe.ch}}

\begin{abstract} 
The one-variable fragment of any first-order logic may be considered as a modal logic, where the universal and existential quantifiers are replaced by a box and diamond modality, respectively. In several cases, axiomatizations of algebraic semantics for these logics have been obtained: most notably, for the modal counterparts $\lgc{S5}$ and $\lgc{MIPC}$ of the one-variable fragments of first-order classical logic and intuitionistic logic, respectively. Outside the setting of first-order intermediate logics, however, a general approach is lacking. This paper provides the basis for such an approach in the setting of first-order lattice-valued logics, where formulas are interpreted in algebraic structures with a lattice reduct. In particular, axiomatizations are obtained for modal counterparts of one-variable fragments of a broad family of these logics by generalizing a functional representation theorem of Bezhanishvili and Harding for monadic Heyting algebras. An alternative proof-theoretic proof is also provided for one-variable fragments of first-order substructural logics that have a cut-free sequent calculus and admit a certain bounded interpolation property.
\end{abstract}

  \begin{keyword}
 Modal Logic, Substructural Logics, Lattice-Valued Logics, One-Variable Fragment, Superamalgamation, Sequent Calculus, Interpolation.
  \end{keyword}
 \end{frontmatter}


\section{Introduction}\label{s:introduction}

The {\em one-variable fragment} of any first-order logic --- the valid formulas built using one variable $x$, unary relation symbols, propositional connectives, and quantifiers $\All{x}$ and $\Exi{x}$ --- may be studied as an ``S5-like'' modal logic. Just replace each occurrence of an atom $P(x)$ with a propositional variable $p$, and $\All{x}$ and $\Exi{x}$ with $\bo$ and $\di$, respectively. The first-order semantics typically induces a relational semantics for this modal logic, but an axiomatization of its algebraic semantics may be rather elusive; in particular, an axiomatization of the first-order logic does not directly yield a Hilbert-style axiomatization of (the modal counterpart of) its one-variable fragment.

Such axiomatizations have been obtained in several notable cases. Monadic Boolean algebras~\cite{Hal55} and monadic Heyting algebras~\cites{MV57,Bul66} correspond to the modal counterparts $\lgc{S5}$ and $\lgc{MIPC}$ of the one-variable fragments of first-order classical logic and intuitionistic logic, respectively. More generally, varieties of monadic Heyting algebras corresponding to  modal counterparts of one-variable fragments of first-order intermediate logics  (based on frames with and without constant domains) have been investigated in~\cites{OS88,Suz89,Suz90,Bez98,BH02,CR15,CMRR17,CMRT22}. One-variable fragments of certain first-order many-valued logics have also been studied in some depth. In particular, modal counterparts of the one-variable fragments of first-order \L ukasiewicz logic and Abelian logic correspond to monadic MV-algebras~\cites{Rut59,dNG04,CCVR20} and monadic Abelian $\ell$-groups~\cite{MT20}, respectively. 

In this paper, we take first steps towards a general approach to addressing this axiomatization problem. As a starting point, we introduce in Section~\ref{s:one-variable}  (first-order) one-variable lattice-valued logics, where formulas are interpreted in structures defined over complete $\lang$-lattices: algebraic structures for a given algebraic signature $\lang$ that have a lattice reduct. In Section~\ref{s:algebraic}, we then define an m-$\lang$-lattice to be an $\lang$-lattice expanded with modalities $\bo$ and $\di$ satisfying certain natural equations, and for a class $\K$ of $\lang$-lattices, let $\mK$ denote the class of m-$\lang$-lattices with an $\lang$-lattice reduct in $\K$. (In particular, if $\K$ is the variety of Boolean algebras or Heyting algebras, $\mK$ is the variety of monadic Boolean algebras or monadic Heyting algebras, respectively.) Generalizing previous results in the literature (see, e.g.,~\cites{Bez98,Tuy21}), we obtain a one-to-one correspondence between m-$\lang$-lattices and $\lang$-lattices equipped with a subalgebra that satisfies a certain relative completeness condition. 

Given a variety $\V$ of $\lang$-lattices, equational consequence in the variety $\mV$ always implies consequence in the one-variable lattice-valued logic based on the complete members of $\V$. In Section~\ref{s:functional}, we show that the converse also holds if $\V$ admits regular completions and has the superamalgamation property. The key tool in this proof is a  generalization of Bezhanishvili and Harding's functional representation theorem for monadic Heyting algebras~\cite{BH02}. In Section~\ref{s:substructural}, we show that this theorem applies to certain varieties of \FLe-algebras, yielding axiomatizations of the modal counterparts of one-variable fragments of certain first-order substructural logics. In Section~\ref{s:prooftheory}, we provide an alternative proof of these results for substructural logics by establishing a bounded interpolation property for cut-free sequent calculi for the one-variable fragments. 

Finally, in Section~\ref{s:concluding} we sketch a broader perspective for one-variable lattice-valued logics based on an arbitrary class $\K$ of complete $\lang$-lattices, observing that when $\K$ consists of the complete members of a variety $\V$, the corresponding class of m-$\lang$-lattices need not in general be $\mV$ or even a variety.


\section{One-Variable Lattice-Valued Logics}\label{s:one-variable}

Let $\lang_n$ denote the set of operation symbols of an algebraic signature $\lang$ of arity $n\in\N$, and call $\lang$ {\em lattice-oriented} if $\lang_2$ contains distinct symbols $\mt$ and $\jn$. We will assume throughout this paper that $\lang$ is a fixed lattice-oriented signature. 

An {\em $\lang$-lattice} is an algebraic structure $\alg{A} = \tuple{A,\{\ops^{\alg{A}}\mid n\in\N,\,\ops\in\lang_n\}}$ such that $\tuple{A,\mt^{\alg{A}},\jn^{\alg{A}}}$ is a lattice with order $x\leq^{\alg{A}}y :\Longleftrightarrow x \mt^{\alg{A}} y = x$ and $\ops^{\alg{A}}$ is an $n$-ary operation on $A$ for each $\ops\in\lang_n$ ($n\in\N$). As usual, we omit superscripts when these are clear from the context. 

We call $\alg{A}$ {\em complete} if its lattice reduct  $\tuple{A,\mt,\jn}$ is complete, i.e., $\bigwedge X$ and $\bigvee X$ exist in $A$ for all $X \subseteq A$. Given a class $\K$ of $\lang$-lattices, we denote by $\overline{\K}$ the class of its complete members and say that $\K$ {\em admits regular completions} if for any $\alg{A}\in\K$, there exist a $\alg{B}\in\overline{\K}$ and an embedding $f\colon\alg{A}\to\alg{B}$ that preserves all existing meets and joins of $\alg{A}$.

\begin{example}
The varieties $\cls{BA}$ and $\cls{HA}$ of Boolean algebras and Heyting algebras, respectively, are closed under MacNeille  completions and hence admit regular completions. Although these are the only two non-trivial varieties of Heyting algebras closed under MacNeille completions~\cite{BH04}, a broad family of varieties that provide semantics for substructural logics (see Section~\ref{s:substructural}) also have this property~\cite{CGT12}. Moreover, for a still broader family of varieties, including the variety $\cls{GA}$ of G{\"o}del algebras (Heyting algebras that satisfy the prelinearity axiom $(x\to y)\lor(y\to x) \eq 1$), the class of their subdirectly irreducible members is closed under MacNeille completions~\cite{CGT11}. Note, however, that neither the variety $\cls{MV}$ of MV-algebras nor the class of its subdirectly irreducible members admit regular completions~\cite{GP02}.
\end{example}

First-order formulas with propositional connectives in~$\lang$ can be defined for an arbitrary predicate language as usual  (see,~e.g.,~\cite{CN21}*{Section~7.1}). We restrict our attention here, however, to the set $\ofml(\lang)$ of {\em one-variable $\lang$-formulas} $\f,\p,\x,\dots$, built from a countably infinite set of unary predicates $\{P_i\}_{i\in\N}$, a variable $x$, connectives in $\lang$, and quantifiers $\all,\exi$. Given $\f,\p\in\ofml(\lang)$, we refer to $\f\eq\p$ as an {\em $\ofml(\lang)$-equation} and let  $\f\leq\p$ denote $\f\mt\p\eq\f$.\footnote{To avoid confusion, let us emphasize that $\f\eq\p$ is a primitive syntactic object that relates two formulas and not terms. In some settings (e.g., first-order classical or intuitonistic logic), the validity of  $\f\eq\p$ can be expressed using the validity of a formula such as $\f \leftrightarrow \p$ and we may define semantical consequence between formulas, but this is not always the case.}

Let $\alg{A}$ be a complete $\lang$-lattice. An {\em $\alg{A}$-structure} is an ordered pair $\model{S}=\tuple{S, \mathcal{I}}$ such that $S$ is a non-empty set and $\mathcal{I}(P_i)$ is a map from $S$ to $A$ for every $i\in\N$. For each $u\in S$, we define a map $\semvalue{\cdot}^\model{S}_u\colon\ofml(\lang)\to A$ inductively as follows:
\begin{align*}
\semvalue{P_i(x)}^\model{S}_u & = \mathcal{I}(P_i)(u) && i\in\N\\[2pt]
\semvalue{\ops(\f_1,\dots,\f_n)}^\model{S}_u 
& = {\ops}^\alg{A}(\semvalue{\f_1}^\model{S}_u\!,\,\dots ,\,\semvalue{\f_n}^\model{S}_u)&& n\in\N,\,\ops\in\lang_n\\[2pt]
\semvalue{\All{x}\f}^\model{S}_u & = \bigwedge \{\semvalue{\f}^{\model{S}}_{v} \mid v\in S\}\\[2pt]
\semvalue{\Exi{x}\f}^\model{S}_u & = \bigvee \{\semvalue{\f}^{\model{S}}_{v} \mid v\in S\}.
\end{align*}
We say that an $\ofml(\lang)$-equation $\f\eq\p$ is {\em valid} in $\model{S}$, denoted by $\model{S}\models\f\eq\p$, if $\semvalue{\f}^\model{S}_u=\semvalue{\p}^\model{S}_u$ for all $u\in S$. We also say that an $\ofml(\lang)$-equation $\f\eq\p$ is a {\em (sentential) semantical consequence}  of a set of $\ofml(\lang)$-equations $T$ with respect to a class $\K$ of complete $\lang$-lattices, denoted by $T\fosc{\K}\f\eq\p$, if $\model{S}\models\f\eq\p$ for any $\alg{A}\in\K$ and $\alg{A}$-structure  $\model{S}$ satisfying $\model{S}\models\f'\eq\p'$ for all $\f'\eq\p' \in T$.\footnote{These notions can be extended to arbitrary (classes of) $\lang$-lattices by saying that $\semvalue{\All{x}\f}^\model{S}_u$ or $\semvalue{\Exi{x}\f}^\model{S}_u$ is {\em undefined} if the corresponding infimum or supremum fails to exist and that $\model{S}$ is \emph{safe} if $\semvalue{\f}^\model{S}_u$ is defined for all $\f\in\ofml(\lang)$ and $u\in S$. Clearly, if $\K$ is a class of complete $\lang$-lattices, then the two notions of semantical consequence coincide, and if $\K$ admits regular completions, then ${\vDash^\forall_\K} = {\vDash^\forall_{\overline{\K}}}$.}

Now let $\mfml(\lang)$ be the set of propositional formulas $\a,\b,\dots$ constructed using propositional variables $\{p_i\}_{i\in\N}$, connectives in $\lang$, and unary connectives $\bo,\di$, and call $\a\eq\b$ an {\em $\mfml(\lang)$-equation} for any $\a,\b\in\mfml(\lang)$. The standard translation functions $({-})^\ast$ and $({-})^\circ$ between $\ofml(\lang)$ and $\mfml(\lang)$ are defined inductively as follows:
\begin{align*}
(P_i(x))^\ast			&= p_i						 & p_i^\circ		&= P_i(x) \\
(\ops(\f_1,\dots,\f_n))^\ast	&= \ops(\f^\ast_1,\dots,\f^\ast_n)	\ \   & (\ops(\a_1,\dots,\a_n))^\circ		&= \ops(\a^\circ_1,\dots,\a^\circ_n) & \ops\in\lang_n\\
(\All{x} \f)^\ast			&=\bo \f^\ast					& (\bo\a)^\circ	&= \All{x} \a^\circ \\
(\Exi{x} \f)^\ast			&=\di \f^\ast					& (\di\a)^\circ	&= \Exi{x} \a^\circ.
\end{align*}
These translations extend in the obvious way also to (sets of) $\ofml(\lang)$-equations and $\mfml(\lang)$-equations.  

Clearly $(\f^\ast)^\circ = \f$ for any $\f\in\ofml(\lang)$ and $(\a^\circ)^\ast = \a$ for any $\a\in\mfml(\lang)$; hence we can alternate between first-order and modal notations as convenient. In particular, given any class $\K$ of complete $\lang$-lattices, we obtain an equational consequence relation on $\mfml(\lang)$ corresponding to $\fosc{\K}$. Therefore, to find an algebraic semantics for a one-variable lattice-valued logic, we seek a (natural) axiomatization of a variety $\V$  of algebras in the signature of $\mfml(\lang)$ such that $\fosc{\K}$ corresponds, via the above translations, to equational consequence in~$\V$. More precisely, let us call a homomorphism from the formula algebra with universe $\mfml(\lang)$ to $\alg{A}\in\V$ an {\em $\alg{A}$-evaluation}, and define for any set $\Si\cup\{\a\eq\b\}$ of $\mfml(\lang)$-equations,
\begin{align*}
\Si\vDash_{\V} \a\eq\b \: :\Longleftrightarrow \enspace
& \text{$f(\a)=f(\b)$ for every $\alg{A}\in\V$ and $\alg{A}$-evaluation $f$}\\
&   \text{such that $f(\a')=f(\b')$ for all $\a'\eq\b'\in\Si$.}
\end{align*}
Then $\V$ should satisfy for any set of $\ofml(\lang)$-equations $T\cup\{\f\eq\p\}$,
\[
T\fosc{\K}\f\eq\p\enspace \iff \enspace T^\ast\vDash_\V\f^\ast\eq\p^\ast.
\]
In Section~\ref{s:functional}, we solve this problem for the case where $\K$ consists of the complete members of a variety of $\lang$-lattices that admits regular completions and has the superamalgamation property (Corollary~\ref{c:completeness}).

\begin{example}
If $\K$ is $\overline{\cls{BA}}$ or $\overline{\cls{HA}}$, then $\fosc{\K}$ is semantical consequence in the one-variable fragment of first-order classical logic or intuitionistic logic, and corresponds to equational consequence in monadic Boolean algebras~\cite{Hal55} or monadic Heyting algebras~\cites{MV57,Bul66}, respectively. Similarly, if $\K$ is $\overline{\cls{GA}}$, then $\fosc{\K}$ is  semantical consequence in the one-variable fragment of the first-order logic of linear frames~\cite{Cor92}, which corresponds to equational consequence in prelinear monadic Heyting algebras~\cite{CMRT22}. On the other hand, if $\K$ is the class of totally ordered members of $\overline{\cls{GA}}$, then $\fosc{\K}$ is semantical consequence in the one-variable fragment of first-order G{\"o}del logic, the first-order logic of linear frames with constant domains, which corresponds to equational consequence in monadic G{\"o}del algebras, i.e., prelinear monadic Heyting algebras satisfying the constant domain axiom $\bo(\bo x\jn y)\eq\bo x\jn\bo y$~\cite{CR15}. Finally, semantical consequence in the one-variable fragment of first-order \L ukasiewicz logic is obtained by taking $\K$ to be the class of totally ordered members of $\overline{\cls{MV}}$ and corresponds to equational consequence in monadic MV-algebras~\cites{Rut59,dNG04,CCVR20}.
\end{example}


\section{An Algebraic Semantics}\label{s:algebraic}

An \emph{m-lattice} is an algebraic structure $\tuple{L,\mt,\jn,\bo,\di}$ such that $\tuple{L,\mt,\jn}$ is a lattice and the following equations are satisfied:
\[
\begin{array}{r@{\quad}l@{\qquad\qquad}r@{\quad}l}
{\rm (L1_\bo)}	&\bo x \mt x \eq\bo x 		& {\rm (L1_\di)}	 & \di x \jn x \eq\di x\\
{\rm (L2_\bo)}	&\bo(x \mt y) \eq\bo x \mt\bo y	& {\rm (L2_\di)}	 &\di (x\jn y)\eq\di x \jn\di y\\
{\rm (L3_\bo)}	&\bo\di x \eq\di x			& {\rm (L3_\di)}	 &\di\bo x \eq\bo x. 
\end{array}
\]
Recalling that $\a\le\b$ stands for $\a\mt\b\eq\a$ and implies $\a\jn\b\eq\b$ in any lattice, it is easy to check that every m-lattice satisfies the following (quasi-)equations:
\[
\begin{array}{r@{\quad}l@{\qquad\qquad}r@{\quad}l}
{\rm (L4_\bo)}	& \bo\bo x \eq\bo x 	& {\rm (L4_\di)} & \di\di x \eq\di x \\
{\rm (L5_\bo)}	& x \le y \,\Longrightarrow\,\bo x \le\bo y 	& {\rm (L5_\di)} 	& x \le y \,\Longrightarrow\,\di x \le\di y.
\end{array}
\]
Let $\lang$ again be a fixed lattice-oriented signature. An {\em m-$\lang$-lattice} is an algebraic structure $\tuple{\alg{A},\bo,\di}$ such that $\alg{A}$ is an $\lang$-lattice, $\tuple{A,\mt,\jn,\bo,\di}$ is an m-lattice, and the following equation is satisfied for each $\ops\in\lang_n$ ($n\in\N$):
\[
\begin{array}{rl}
(\ops_\bo)	& \quad \bo(\ops(\bo x_1,\dots,\bo x_{n})) \eq \ops(\bo x_1,\dots,\bo x_{n}).
\end{array}
\]
It follows from (\ops$_\bo$), {\rm (L3$_\bo$)}, and {\rm (L3$_\di$)} that $\tuple{\alg{A},\bo,\di}$ also satisfies the following equation for each $\ops\in\lang_n$ ($n\in\N$):
\[
\begin{array}{rl}
(\ops_\di)	& \quad \di(\ops(\di x_1,\dots,\di x_{n})) \eq \ops(\di x_1,\dots,\di x_{n}).
\end{array}
\]
Given any class $\K$ of $\lang$-lattices, we let $\mK$ denote the class of all m-$\lang$-lattices $\tuple{\alg{A},\bo,\di}$ such that $\alg{A}\in\K$. Clearly, if $\K$ is a variety, then also $\mK$ is a variety. 

\begin{example}\label{e:monadicvarieties}
It is easily checked that $\mcls{BA}$ and $\mcls{HA}$ are the varieties of monadic Boolean algebras~\cite{Hal55} and monadic Heyting algebras~\cites{MV57}, respectively. Similarly, $\mcls{GA}$ is the variety of prelinear monadic Heyting algebras~\cite{CMRT22}, while the subvariety of $\mcls{GA}$  satisfying the constant domain equation is the variety of monadic G{\"o}del algebras~\cite{CR15}. However, $\mcls{MV}$ is not the variety of monadic MV-algebras considered in~\cites{Rut59,dNG04,CCVR20}, which satisfy the equation $\di x \pd \di x \eq \di (x \pd x)$. To see this, consider the MV-algebra $\textbf{\L}_3 = \tuple{\{0,\frac12,1\},\mt,\jn,\pd,\to,0,1}$ (in the language of \FLe-algebras) with the usual order, where $x\pd y:=\max(0,x+y-1)$, $x\to y:=\min(1,1-x+y)$. Let $\bo 0 = \bo\frac12= \di 0 = 0$ and $\bo 1=\di\frac12=\di 1 =1$. Then $\tuple{\textbf{\L}_3,\bo,\di}\in\mcls{MV}$, but $\di\frac12\pd\di\frac12= 1\pd 1=1\neq 0=\di 0=\di(\frac12\pd\frac12)$. 
\end{example}

We now establish a useful representation theorem for m-$\lang$-lattices that will be crucial in the proof of the functional representation theorem for certain varieties in the next section.

\begin{lemma}\label{l:subalgebra}
Let $\tuple{\alg{A},\bo,\di}$ be any m-$\lang$-lattice. Then $\bo A:=\{\bo a\mid a\in A\}$ is a subuniverse of $\alg{A}$ satisfying $\bo A = \di A:= \{\di a\mid a\in A\}$ and for any $a\in A$,
\[
\bo a = \max \{b\in\bo A\mid b\le a\} \quad\text{ and }\quad\di a = \min\{b\in\bo A\mid a\le b\}.
\]
We let $\bo\alg{A}$ denote the subalgebra of $\alg{A}$ with universe $\bo A$. 
\end{lemma}
\begin{proof}
The fact that $\bo A$ is a subuniverse of $\alg{A}$ is a direct consequence of (\ops$_\bo$), and the fact that $\bo A =\di A$ follows from {\rm (L3$_\bo$)} and {\rm (L3$_\di$)}. Moreover, if $b\in\bo A$ satisfies $b \le a$, then $b=\bo b\le \bo a$, by {\rm (L4$_\bo$)} and {\rm (L5$_\bo$)}. But also $\bo a\le a$, by {\rm (L1$_\bo$)}, so $\bo a = \max \{b\in\bo A\mid b\le a\}$. Similarly, $\di a = \min\{b\in\bo A\mid a\le b\}$.
\end{proof}

A sublattice $\alg{L}_0$ of a lattice $\alg{L}$ is said to be {\em relatively complete} if for any $a\in L$, the set $\{b\in L_0\mid b\le a\}$ contains a maximum and the set $\{b\in L_0\mid a\le b\}$ contains a minimum. Equivalently, $\alg{L}_0$ is relatively complete if the inclusion map $f_0$ from  $L_0$ to $L$ has left and right adjoints: that is, there exist order-preserving maps $\bo\colon L\to L_0$, $\di\colon L\to L_0$ satisfying for $a\in L$, $b\in L_0$,
\[
f_0(b) \le a \iff b \le\bo a \quad\text{ and }\quad a \le f_0(b) \iff\di a\le b.
\]
For convenience, we also say that a subalgebra $\alg{A}_0$ of an $\lang$-lattice $\alg{A}$ is relatively complete if this is the case for the lattice reducts. By Lemma~\ref{l:subalgebra}, the subalgebra $\bo\alg{A}$ of $\alg{A}$ is relatively complete for any m-$\lang$-lattice  $\tuple{\alg{A},\bo,\di}$. The following result establishes a converse.

\begin{lemma}\label{l:converse}
Let $\alg{A}_0$ be a relatively complete subalgebra of an $\lang$-lattice $\alg{A}$, and define $\bo_0 a := \max \{b\in A_0\mid b\le a\}$ and $\di_0 a := \min\{b\in A_0\mid a\le b\}$ for each $a\in A$. Then $\tuple{\alg{A},\bo_0,\di_0}$ is an m-$\lang$-lattice and $\bo_0 A =\di_0 A = A_0$.
\end{lemma}
\begin{proof}
It is straightforward to check that $\tuple{A,\mt,\jn,\bo_0,\di_0}$ is an m-lattice; for example, it satisfies {\rm (L2$_\bo$)}, since for any $a_1,a_2\in A$,
\begin{align*}
\bo_0 (a_1\mt a_2) 
& = \max \{b\in A_0\mid b\le a_1\mt a_2\}\\[2pt]
& = \max \{b\in A_0\mid b\le a_1 \text{ and } b \le a_2\} \\[2pt]
& = \max \{b\in A_0\mid b\le a_1\} \mt \max \{b\in A_0\mid b\le a_2\} \\[2pt]
& = \bo_0 a_1 \mt \bo_0 a_2.
\end{align*}
Since $\alg{A}_0$  is a subalgebra of $\alg{A}$, clearly $\tuple{\alg{A},\bo_0,\di_0}$ also satisfies (\ops$_\bo$). Hence $\tuple{\alg{A},\bo_0,\di_0}$ is an m-$\lang$-lattice and $\bo_0 A =\di_0 A = A_0$.
\end{proof}

Combining Lemmas~\ref{l:subalgebra} and~\ref{l:converse}, we obtain the following representation theorem for  m-$\lang$-lattices.

\begin{theorem}\label{t:correspondence}
Let $\K$ be any class of $\lang$-lattices. Then there exists a one-to-one correspondence between the members of $\mK$ and ordered pairs $\tuple{\alg{A},\alg{A}_0}$, where $\alg{A}\in\K$ and $\alg{A}_0$ is a relatively complete subalgebra of $\alg{A}$, implemented by the maps $\tuple{\alg{A},\bo,\di}\mapsto\tuple{\alg{A},\bo\alg{A}}$ and $\tuple{\alg{A},\alg{A}_0}\mapsto\tuple{\alg{A},\bo_0,\di_0}$.
\end{theorem}

We now show that m-$\lang$-lattices encompass the algebraic semantics of the one-variable lattice-valued logics defined in Section~\ref{s:one-variable}.

\begin{proposition}\label{p:standard}
Let $\alg{A}$ be any complete $\lang$-lattice and let $W$ be any set. Then $\tuple{\alg{A}^W,\bo,\di}$ is an m-$\lang$-lattice, where the operations of $\alg{A}^W$ are defined pointwise and for each $f\in A^W$ and $u\in W$, 
\[
\bo f(u) = \bigwedge_{v\in W}f(v)\quad\text{ and }\quad\di f(u) = \bigvee_{v\in W}f(v).
\]
Moreover, if $\alg{A}\in\V$ for some variety $\V$ of $\lang$-lattices, then $\tuple{\alg{A}^W,\bo,\di}\in\mV$.
\end{proposition}
\begin{proof}
Since $\alg{A}$ is an $\lang$-lattice, $\alg{A}^W$, with operations defined pointwise, is also an $\lang$-lattice. It is also easy to check that $\tuple{A^W,\mt,\jn,\bo,\di}$ is an m-lattice; for example, $\tuple{A^W,\mt,\jn,\bo,\di}$ satisfies {\rm (L2$_\bo$)}, since for any $f,g\in A^W$ and $u\in W$,
\[
\bo(f\mt g)(u) 
= \bigwedge_{v\in W}(f\mt g)(v)
= (\bigwedge_{v\in W}f(v))\mt (\bigwedge_{v\in W}g(v))
 = (\bo f \mt \bo g)(u).
\]
Moreover, for any $\ops\in\lang_n$ ($n\in\N$), $f_1,\dots,f_n\in A^W$, and $u\in W$,
\begin{align*}
\bo(\mathop{\ops}(\bo f_1,\dots,\bo f_{n}))(u)
 & = \bigwedge_{v\in W}\ops(\bo f_1,\dots,\bo f_{n})(v)\\[2pt]
 & = \bigwedge_{v\in W}\ops(\bo f_1(v),\dots,\bo f_{n}(v))\\[2pt]
 & = \ops(\bo f_1(u),\dots,\bo f_{n}(u))\\[3pt]
  & = \ops(\bo f_1,\dots,\bo f_n)(u),
 \end{align*}
noting that in the last-but-one equality we have used the fact that $\bo f_i(v) = \bo f_i(u)$ for each $v \in W$ and $i\in\{1,\dots,n\}$. Hence $\tuple{\alg{A}^W,\bo,\di}$ satisfies (\ops$_\bo$). Finally, if $\alg{A}\in\V$ for some variety $\V$ of $\lang$-lattices, then, since varieties are closed under taking direct products, also $\alg{A}^W\in\V$, and so $\tuple{\alg{A}^W,\bo,\di}\in\mV$.
\end{proof}

Let us call an m-$\lang$-lattice $\tuple{\alg{A}^W,\bo,\di}$ defined as described in Proposition~\ref{p:standard} {\em full functional},  and say that an m-$\lang$-lattice is {\em functional} if it is isomorphic to a subalgebra of a full functional m-$\lang$-lattice. The following identification of  the semantics of one-variable lattice-valued logics with evaluations into full functional m-$\lang$-lattices is a direct consequence of  Proposition~\ref{p:standard}.

\pagebreak

\begin{corollary}
Let $\alg{A}$ be any complete $\lang$-lattice. 
\begin{enumerate}
\item[\rm (a)]
Given any $\alg{A}$-structure $\model{S}=\tuple{S, \mathcal{I}}$, the evaluation $f$ for the  full functional m-$\lang$-lattice $\tuple{\alg{A}^S,\bo,\di}$ defined by $f(p_i) := \mathcal{I}(P_i)$ for each $i\in\N$ satisfies for all $\f,\p\in \ofml(\lang)$ and $u\in S$,
\[
f(\f^\ast)(u)= \semvalue{\f}^\model{S}_u \quad\text{ and }\:\quad \model{S}\models \f\eq \p \iff f(\f^\ast) = f(\p^\ast).
\]

\item[\rm (b)] 
Given any evaluation $g$ for a full functional m-$\lang$-lattice $\tuple{\alg{A}^W,\bo,\di}$, the $\alg{A}$-structure $\model{W}=\tuple{W, \mathcal{J}}$, where $\mathcal{J}(P_i):=g(p_i)$ for each $i\in\N$,  satisfies for all $\f,\p\in \ofml(\lang)$ and $u\in W$,
\[
g(\f^\ast)(u)  = \semvalue{\f}^\model{W}_u \quad\text{ and }\:\quad \model{W}\models \f\eq \p \iff g(\f^\ast) = g(\p^\ast).
\]
\end{enumerate}
\end{corollary}

\begin{corollary}\label{c:soundness2}
For any variety $\V$ of $\lang$-lattices and set of $\ofml(\lang)$-equations $T\cup\{\f\eq\p\}$,
\[
T^\ast\vDash_{\mV}\f^\ast\eq\p^\ast \quad\Longrightarrow\quad T\fosc{\overline{\V}}\f\eq\p.
\]
Moreover, the converse also holds if every member of $\mV$ is functional. 
\end{corollary}


\section{A Functional Representation Theorem}\label{s:functional}

In this section, we establish a functional representation theorem for $\mK$ when $\K$ is a class of $\lang$-lattices satisfying certain conditions, following very closely a proof of the same result for monadic Heyting algebras~\cite{BH02}*{Theorem~3.6}. 

Let $\K$ be any class of $\lang$-lattices. A {\em V-formation} in $\K$ is a $5$-tuple $\tuple{\alg{A},\alg{B}_1,\alg{B}_2,f_1,f_2}$ consisting of $\alg{A},\alg{B}_1,\alg{B}_2\in \K$ and embeddings $f_1\colon\alg{A}\to\alg{B}_1$, $f_2\colon\alg{A}\to\alg{B}_2$. An {\em amalgam} in $\K$ of this V-formation  is a triple $\tuple{\alg{C},g_1,g_2}$ consisting of $\alg{C}\in\K$ and embeddings $g_1\colon\alg{B}_1\to\alg{C}$, $g_2\colon\alg{B}_2\to\alg{C}$ such that $g_1\circ f_1 = g_2\circ f_2$. It is called a {\em superamalgam} if for any $b_1\in B_1$, $b_2\in B_2$ and distinct $i,j\in\{1,2\}$ such that $g_i(b_i)\le g_j(b_j)$, there exists an $a\in A$ such that $g_i(b_i)\le g_i\circ f_i(a) = g_j\circ f_j(a)  \le g_j(b_j)$. The class $\K$ is said to have the {\em superamalgamation property} if any V-formation in $\K$ has a superamalgam in~$\K$.
 
\begin{theorem}\label{t:functional}
Let $\K$ be a class of $\lang$-lattices that is closed under subalgebras and direct limits, admits regular completions, and has the superamalgamation property. Then any member of $\mK$ is functional.
\end{theorem}

\begin{proof}
Consider any $\tuple{\alg{A},\bo,\di}\in\mK$. Then $\alg{A}\in\K$ and, since $\K$ is closed under subalgebras, also $\bo\alg{A}\in\K$. We let $W$ denote the set of positive natural numbers and define inductively a sequence of $\lang$-lattices $\tuple{\alg{A}_i}_{i\in W}$ in $\K$ and sequences of embeddings $\tuple{f_i\colon \bo\alg{A} \to \alg{A}_i}_{i\in W}$, $\tuple{g_i\colon \alg{A} \to \alg{A}_{i}}_{i\in W}$, and $\tuple{s_i\colon \alg{A}_{i-1} \to \alg{A}_{i}}_{i\in W}$. Let $\alg{A}_0 := \alg{A}$. By assumption, there exists for each $i\in W$, a superamalgam $\tuple{\alg{A}_i, s_i, g_i}$ of the V-formation $\tuple{\bo\alg{A}, \alg{A}_{i-1},\alg{A}, f_{i-1}, f_0}$, where  $f_0\colon\bo A\to A$  is the inclusion map and $f_i := s_i\circ f_{i-1} = g_{i}\circ f_0  = g_{i}|_{\bo A}$.

Now let $\alg{L}$ be the direct limit of the system $\tuple{\tuple{\alg{A}_i, s_{i+1}}}_{i\in W}$ with associated embeddings
$\tuple{l_i\colon \alg{A}_i \to \alg{L}}_{i\in W}$. Then, by assumption, $\alg{L}\in\K$, and there exist a complete $\lang$-lattice $\alg{\overline{L}}\in\K$ and an embedding $h\colon\alg{L}\to\alg{\overline{L}}$ that preserves all existing meets and joins of $\alg{L}$. We depict the first two amalgamation steps of this construction in the following diagram:

\begin{center}
\begin{tikzpicture}
\coordinate[label=below:$\Box\mathbf{A}$](Box A1) at (0,0.5);
\coordinate[label=below:$\mathbf{A}$](Box A2) at (2.5,0.5);
\coordinate[label=below:$\mathbf{A}$](A) at (0,-1.6);
\coordinate[label=below:$\mathbf{A}_1$](A1) at (2.6,-1.6);
\coordinate[label=below:$\mathbf{A}_2$](A2) at (5.2,-1.6);
\coordinate[label=below:$\mathbf{A}_3$](A3) at (7.7,-1.6);
\coordinate[label=below:$\mathbf{L}$] (L) at (5.2,-3.7);
\coordinate[label=below:$\mathbf{\overline{L}}$] (CL) at (7.7,-3.67);
\coordinate[label=below:$\cdots$] (dots) at (8.5,-1.7);
\draw[->] (0,-0.2) -- (0,-1.47) node[midway, right]{$f_0$};
\draw[->] (0.5,0.2) -- (2.1,0.2) node[midway, above]{$f_0$};
\draw[->] (0.5,-0.2) -- (2.1,-1.47) node[midway, above right]{$f_1$};
\draw[->] (3,0.2) -- (7.2,-1.47) node[midway, above right]{$g_3$};
\draw[->] (3,-0.2) -- (4.7,-1.47) node[midway, above right]{$g_2$};
\draw[->] (2.5,-0.2) -- (2.5,-1.47) node[midway, right]{$g_1$};
\draw[->] (0.5,-1.88) -- (2.1,-1.88) node[midway, above]{$s_1$};
\draw[->] (3,-1.88) -- (4.7,-1.88) node[midway, above]{$s_2$};
\draw[->] (5.6,-1.88) -- (7.2,-1.88) node[midway, above]{$s_3$};
\draw[->] (3,-2.25) -- (4.85,-3.65) node[midway, above right]{$l_1$};
\draw[->] (5.2,-2.25) -- (5.2,-3.65) node[midway, right]{$l_2$};
\draw[->] (7.7,-2.25) -- (5.6,-3.65) node[midway, above left]{$l_3$};
\draw[->] (5.6,-4.0) -- (7.2,-4.0) node[midway, above]{$h$};
\end{tikzpicture}
\end{center}
To show that $\tuple{\alg{A},\bo,\di}$ is functional, it suffices to prove that the following map into the full functional m-$\lang$-lattice $\tuple{\alg{\overline{L}}^W, \bo,\di}$ is an embedding:
$$
f\colon \tuple{\alg{A},\bo,\di} \to \tuple{\alg{\overline{L}}^W, \bo,\di}; \quad a\mapsto \tuple{h\circ l_i\circ g_{i}(a)}_{i\in W}.
$$
It is easy to see that $f$ is an embedding  of $\lang$-lattices, so it remains to show that $f(\bo a)  = \bo f(a)$ and $f(\di a)  = \di f(a)$ for all $a\in A$. Fix some $a\in A$. Considering just the case of $\bo$ (since the case of $\di$ is analogous), we obtain
\begin{align*}
f(\bo a) & = \tuple{h\circ l_i\circ g_i(\bo a)}_{i\in W}\\[2pt]
&\mathbin{\stackrel{(1)}{=}} \tuple{h(\bigwedge\limits_{j\in W} l_j\circ g_j(a))}_{i\in W} \\
& \mathbin{\stackrel{(2)}{=}} \tuple{\bigwedge\limits_{j\in W} h\circ l_j\circ g_j(a) }_{i\in W}\\[2pt]
& = \bo \tuple{h\circ l_i\circ g_i(a)}_{i\in W}\\[2pt]
& =  \bo f(a).
\end{align*}
To justify (1),  it suffices to fix an $i\in W$ and show that  $l_i\circ g_i(\bo a)$ is the greatest lower bound of the set $S := \{ l_j\circ g_j(a) \mid j\in W\}$. This implies that $S$ has an infimum and so also (2) follows by the definition of regular completions. 

We start by showing that $l_i\circ g_i(\bo a) = l_j\circ g_j(\bo a)$ for all $j\in W$. Clearly, this follows from the fact that for any $k\in W$,
$$
l_k\circ g_k(\bo a) = l_k\circ f_k(\bo a) = l_{k+1} \circ s_{k+1} \circ f_k(\bo a) =   l_{k+1} \circ g_{k+1} (\bo a),
$$
where the first and the last equality are due to the definition of $f_k$ and the second is due to the definition of direct limits. Using this fact, we can easily show that $l_i\circ g_i(\bo a)$ is a lower bound of $S$: just observe that since $\bo a \leq a$, for each $j\in W$,
\[
l_i\circ g_i(\bo a)  = l_j\circ g_j(\bo a)  \leq l_j\circ g_j( a).
\]
 Finally, suppose that $c\in L$ is any lower bound of $S$. By the definition of direct limits, there exist a $k\in W$ and a $d\in A_k$ such that
 \[
  l_{k+1} \circ s_{k+1} (d) =  l_k (d)  = c \le  l_{k+1} \circ g_{k+1} (a).
 \]
Moreover, since $l_{k+1}$ is an embedding, $s_{k+1} (d) \leq g_{k+1}(a)$.  Therefore, because $\tuple{\alg{A}_{k+1}, s_{k+1}, g_{k+1}}$ is a superamalgam of the V-formation $\tuple{\bo\alg{A}, \alg{A}_k,\alg{A}, f_k, f_0}$, there exists a $b\in \bo A$ such that
$$
s_{k+1} (d) \leq s_{k+1} \circ f_k(b) =  g_{k+1}\circ f_0(b) \leq g_{k+1}(a).
$$
Since $s_{k+1}$ and $g_{k+1}$ are embeddings and $f_0$ is the inclusion map, $d \leq f_k(b)$ and $b\leq a$. The latter inequality together with $b\in \bo A$ entails that $b\leq \bo a$. Hence also $f_k(b) \leq f_k(\bo a) = g_k(\bo a)$ and, using the first inequality, 
\begin{align*}
c & = l_k(d) \le l_k\circ f_k(b) \leq l_k\circ g_k(\bo a)  = l_i\circ g_i(\bo a). 
\end{align*}
So $l_i\circ g_i(\bo a)$ is the greatest lower bound of the set $S$ as required.
\end{proof}

Combining Theorem~\ref{t:functional} with Corollary~\ref{c:soundness2} yields the following result.

\begin{corollary}\label{c:completeness}
Let $\V$ be a variety of $\lang$-lattices that admits regular completions and has the superamalgamation property. Then for any set $T\cup\{\f\eq\p\}$ of $\ofml(\lang)$-equations,
\[
T\fosc{\overline{\V}}\f\eq\p \quad\Longleftrightarrow\quad T^\ast\vDash_{\mV}\f^\ast \eq\p^\ast.
\]
\end{corollary}

\begin{example}\label{e:lattices}
The variety of lattices admits regular completions and has the superamalgamation property~\cite{Gra98}. Hence, by Theorem~\ref{t:functional}, every m-lattice is functional, and, by Corollary~\ref{c:completeness},  consequence in the one-variable fragment of first-order lattice logic corresponds to equational consequence in m-lattices.
\end{example}


\section{One-Variable Substructural Logics}\label{s:substructural}

In this section, we turn our attention to one-variable fragments of first-order substructural logics. Let $\langs$ be the lattice-ordered signature consisting of binary connectives $\jn$, $\mt$, $\pd$, and $\to$, and constant symbols $\zr$ and $\ut$. An {\em \FLe-algebra} is an $\langs$-lattice $\alg{A} = \tuple{A,\jn,\mt,\pd,\to,\zr,\ut}$ such that $\tuple{A,\pd, \ut}$ is a commutative monoid and $\to$ is the residuum of $\pd$, i.e., $a \pd b \leq c \iff a\leq b\to c$\, for all $a,b,c\in A$.

The class of \FLe-algebras forms a variety $\cls{FL}_e$ that serves as an algebraic semantics for the Full Lambek Calculus with exchange (see, e.g.,~\cite{GJKO07}), and subvarieties of $\cls{FL}_e$ provide algebraic semantics for other substructural logics. In particular, the varieties $\cls{FL}_{ew}$ and $\cls{FL}_{ec}$ consist of \FLe-algebras satisfying $\zr\le x\le \ut$ and $x\le x\pd x$, respectively. The variety $\cls{FL}_{ew}\cap\cls{FL}_{ec}$ is term-equivalent to $\cls{HA}$ (just identify $\pd$ and $\mt$), and $\cls{BA}$ and $\cls{GA}$ are term-equivalent to the subvarieties of $\cls{FL}_{ew}\cap\cls{FL}_{ec}$ axiomatized by $(x\to\zr)\to\zr\eq x$ and $(x\to y)\jn(y\to x)\eq\e$, respectively. $\cls{MV}$  is term-equivalent to  the subvariety of $\cls{FL}_{ew}$ satisfying $x\jn y \eq (x\to y)\to y$.

In combination with residuation, the defining equations of m-$\langs$-lattices yield further relationships between the propositional and modal connectives.

\begin{proposition}\label{p:mFLe}
Let $\alg{A}$ be an \FLe-algebra. Then any m-$\langs$-lattice $\tuple{\alg{A},\bo,\di}$ satisfies the equations
\[
\begin{array}{r@{\quad}l@{\qquad\qquad}r@{\quad}l}
{\rm (L6_\bo)}	&\bo (x\to \bo y) \eq\di x \to\bo y	& {\rm (L6_\di)}	& \bo (\bo x\to y) \eq\bo x \to\bo y.
\end{array}
\]
That is,  $\tuple{\alg{A},\bo,\di}$  is a monadic \FLe-algebra in the sense of~\cite{Tuy21}*{Definition~2.1}.
\end{proposition}
\begin{proof}
We just prove {\rm (L6$_\bo$)}, as the proof for {\rm (L6$_\di$)} is very similar. Consider any $a,b\in A$. Since $a\le\di a$, by {\rm (L1$_\di$)}, also  $\di a\to\bo b\le a\to\bo b$. Hence, using {\rm (L5$_\bo$)}, {\rm (L3$_\bo$)}, and {\rm ($\ops_\bo$)},
\[
\di a\to\bo b = \bo \di a \to \bo b = \bo(\bo\di a \to\bo b) = \bo(\di a \to \bo b)\le \bo(a\to\bo b).
\]
Conversely, since $\bo (a\to\bo b) \le a\to\bo b$, by {\rm (L1$_\bo$)}, it follows by residuation that $a \le\bo (a\to\bo b)\to\bo b$ and hence, using {\rm (L5$_\di$)}, {\rm (L3$_\di$)}, and (\ops$_\di$),
\[
\di a \le\di(\bo (a\to\bo b)\to\bo b)=\bo (a\to\bo b)\to\bo b.
\] 
By residuation again, $\bo (a\to\bo b) \le\di a \to\bo b$. 
\end{proof}

The varieties $\cls{FL}_e$, $\cls{FL}_{ew}$, and $\cls{FL}_{ec}$ are closed under MacNeille completions and have the superamalgamation property (see,~e.g.,~\cite{GJKO07}). Hence Theorem~\ref{t:functional} and Corollary~\ref{c:completeness} yield the following result. 

\begin{theorem}\label{t:FLecompleteness}
Let $\V\in\{\cls{FL}_e,\cls{FL}_{ew},\cls{FL}_{ec}\}$. 
\begin{enumerate}
\item[\rm (a)]	Any member of $\mV$ is functional.  
\item[\rm (b)]	For any set $T\cup\{\f\eq\p\}$ of $\ofml(\langs)$-equations,
\[
T\fosc{\overline{\V}}\f\eq\p \quad\Longleftrightarrow\quad T\vDash_{\mV}\f^\ast\eq\p^\ast.
\]
\end{enumerate}
\end{theorem}

\noindent
In~\cite{CGT12} it was proved that a variety of \FLe-algebras axiomatized relative to $\cls{FL}_e$ by equations of a certain simple syntactic form (called ``$\mathcal{N}_2$-equations'') is closed under MacNeille completions if and only if it has an analytic sequent calculus of a certain form. It has also been proved that many varieties of \FLe-algebras  have the superamalgamation property (equivalently, the Craig interpolation property) (see,~e.g.,~\cites{GJKO07,TabJal22}), but a precise characterization is not known.


\section{A Proof-Theoretic Approach}\label{s:prooftheory}

In this section, we apply proof-theoretic methods to obtain an alternative proof of Theorem~\ref{t:FLecompleteness}(b). Although no new results are proved, the approach described here may be used to deal with varieties of $\lang$-lattices that either do not admit regular completions or lack the superamalgamation property, and may also be used to tackle decidability and complexity issues for one-variable lattice logics.

Let $\ofmls(\langs)$ be the set of first-order $\langs$-formulas $\f,\p,\x,\dots$ constructed inductively from unary predicates $\{P_i\}_{i\in\N}$, variables $\{x\}\cup\{x_i\}_{i\in\N}$, connectives in $\langs$, and quantifiers $\All{x}$ and $\Exi{x}$  such that no variable $x_i$ is in the scope of a quantifier. Clearly, $\ofml(\langs)\subseteq\ofmls(\langs)$. We also write $\f(\bar{y})$ to denote that the free variables of $\f\in\ofmls(\langs)$ belong to the set $\bar{y}$, recalling that a sentence is a formula with no free variables. 

For the purposes of this paper, a {\em sequent} is an ordered pair of finite multisets of first-order $\langs$-formulas, denoted by $\Ga\seq\De$, such that $\De$ contains a most one $\langs$-formula. We also define for $n\in\N^{>0}$ and $\f_1,\dots,\f_n,\p\in\ofmls(\langs)$,
\[
\textstyle
\prod(\f_1,\dots,\f_n):=\f_1\cdots\f_n, \quad
\prod():=\ut, \quad
\sum(\p):=\p, \quad
\sum():=\zr.
\]
The (cut-free) sequent calculus $\fFLe$ is presented in Figure~\ref{f:fFLe}, subject to the following side-conditions:
\begin{enumerate}
\item[\rm (i)] the term $t$ occurring in $\falr$ and $\err$ is a variable occurring in the conclusion of the rule;
\item[\rm (ii)]  the variable $y$ occurring in the premise of $\farr$ and $\elr$ does not occur freely in the conclusion of the rule.
\end{enumerate}
$\fFLew$ and $\fFLec$ are defined as the extensions of $\fFLe$ with, respectively,
\[
\vcenter{\infer[\pfa{\wkr}]{\Ga_1,\Ga_2 \seq \De_1,\De_2}{\Ga_1 \seq \De_1}}
\qquad\text{ and }\qquad
\vcenter{\infer[\pfa{\cnr}]{\Ga_1,\Ga_2 \seq \De}{\Ga_1,\Ga_2,\Ga_2 \seq \De}}
\]
If there exists a derivation $d$ of a sequent $\Ga\seq\De$ in a sequent calculus $\lgc{C}$, we write either $d\der{\lgc{C}}\Ga\seq\De$ or $\der{\lgc{C}}\Ga\seq\De$, and let $\md(d)$ denote the maximum number of applications of $\falr$, $\farr$, $\elr$, $\err$ on a branch of $d$.

\begin{figure}
\centering
\fbox{
 \parbox{.9\linewidth}{
 \[
\begin{array}{c}
\text{Axioms}\\[.1in]
\begin{array}{ccccc}
\infer[\pfa{\idr}]{\f \seq \f}{} & \qquad & \infer[\pfa{\flr}]{\zr \seq}{} & \qquad & \infer[\pfa{\trr}]{\seq\ut}{}\
\end{array}\\[.15in]
\text{Operation Rules}\\[.1in]
\begin{array}{ccc}
\infer[\pfa{\tlr}]{\Ga,  \ut \seq \De}{\Ga \seq \De} & \quad &\infer[\pfa{\frr}]{\Ga \seq\zr}{\Ga \seq}\\[.13in]
\infer[\pfa{\ilr}]{\Ga_1, \Ga_2, \f \to \p \seq \De}{\Ga_1 \seq \f & \Ga_2, \p \seq \De} & \quad & 
\infer[\pfa{\irr}]{\Ga \seq \f \to \p}{\Ga, \f \seq \p}\\[.13in]
\infer[\pfa{\pdlr}]{\Ga,  \f \pd \p \seq \De}{\Ga,\f,\p \seq \De} & \quad &
\infer[\pfa{\pdrr}]{\Ga_1, \Ga_2 \seq \f \pd \p}{\Ga_1 \seq \f & \Ga_2 \seq \p}\\[.13in]
\infer[\pfa{\alr_1}]{\Ga, \f \mt \p \seq \De}{\Ga, \f \seq \De} & & 
\infer[\pfa{\orr_1}]{\Ga \seq \f \jn \p}{\Ga \seq \f}\\[.13in]
\infer[\pfa{\alr_2}]{\Ga, \f \mt \p \seq \De}{\Ga, \p \seq \De} & & 
\infer[\pfa{\orr_2}]{\Ga \seq \f \jn \p}{\Ga \seq \p}\\[.13in]
\infer[\pfa{\olr}]{\Ga, \f \jn \p \seq \De}{\Ga, \f \seq \De & \Ga, \p \seq \De} & & 
\infer[\pfa{\arr}]{\Ga \seq \f \mt \p}{\Ga \seq \f & \Ga \seq \p}\\[.13in]
\infer[\pfa{\falr}]{\Ga, \All{x}\f(x) \seq\De}{\Ga,\f(t)\seq\De} & & 
\infer[\pfa{\farr}]{\Ga\seq \All{x}\p(x)}{\Ga\seq\p(y)}\\[.13in]
\infer[\pfa{\elr}]{\Ga,\Exi{x}\f(x) \seq\De}{\Ga,\f(y)\seq\De} & & 
\infer[\pfa{\err}]{\Ga\seq \Exi{x}\p(x)}{\Ga\seq\p(t)}\\[-.10in]
\end{array}
\end{array}
\]}}
\caption{The Sequent Calculus $\fFLe$}
\label{f:fFLe}
\end{figure}

The next result follows directly from well-known completeness and cut-elimination results for sequent calculi for (first-order) substructural logics.

\begin{proposition}[\cites{OK85,Kom86}]\label{p:ono} 
Let $\V$ be $\cls{FL}_e$, $\cls{FL}_{ew}$, or $\cls{FL}_{ec}$, and let $\lgc{C}$ be $\fFLe$, $\fFLew$, or $\fFLec$, respectively.  For any  sequent $\Ga\seq\De$ consisting only of formulas from $\ofml(\langs)$, 
\[
\textstyle\der{\lgc{C}}\Ga\seq\De \quad\Longleftrightarrow\quad\: \fosc{\overline{\V}} \prod\Ga\le\sum\De.
\]
\end{proposition}

We now prove that each of $\fFLe$, $\fFLew$, and $\fFLec$ satisfies a certain bounded interpolation property. 

\begin{theorem}\label{t:interpolation} 
Let $\lgc{C}\in\{\fFLe,\fFLew,\fFLec\}$. If $d\der{\lgc{C}}\Ga(\bar{y}),\Pi(\bar{z})\seq\De(\bar{z})$, where $\bar{y}\cap\bar{z}=\emptyset$, then there exist a sentence $\x$ and derivations $d_1,d_2$ such that $\md(d_1),\md(d_2)\le\md(d)$, $d_1\der{\lgc{C}}\Ga(\bar{y})\seq\x$, and $d_2\der{\lgc{C}}\Pi(\bar{z}),\x\seq\De(\bar{z})$.
\end{theorem}

\begin{proof}
We proceed by induction on the height of $d$ and consider the last rule applied. Note that if $\bar{y}$ or $\bar{z}$ are empty --- in particular, if $\Ga(\bar{y}),\Pi(\bar{z})\seq\De(\bar{z})$ is an instance of an axiom --- we can take $\x:=\prod\Ga$ or $\x:=\prod\Pi$, respectively, so we may assume that this is not the case. Note also that the additional cases of $\wkr$ for $\fFLew$ and $\cnr$ for $\fFLec$ follow directly from an application of the induction hypothesis. We just consider here the quantifier rules $\falr$ and $\farr$, dealing with the remaining rules in the appendix.
 
\medskip
\noindent $\bullet$ $\falr$: Suppose first that $\Ga(\bar{y})$ is $\Ga'(\bar{y}),\All{x}\f(x)$ and 
\[
 d'\der{\lgc{C}}\Ga'(\bar{y}),\f(u),\Pi(\bar{z})\seq\De(\bar{z}),
\]
where $\md(d')=\md(d)-1$. For subcase (i), suppose that $u\in\bar{y}$. By the induction hypothesis, there exist a sentence $\x$ and derivations $d'_{1},d_{2}$ such that $\md(d'_1),\md(d_2)\le\md(d')$ and
\[
d'_1\der{\lgc{C}}\Ga'(\bar{y}),\f(u)\seq\x, \qquad d_2\der{\lgc{C}}\Pi(\bar{z}),\x\seq\De(\bar{z}).
\]
Using $\falr$, there exists also a derivation $d_1$ such that $\md(d_1)=\md(d'_1)+1\le\md(d')+1=\md(d)$ and
 \[
 d_1\der{\lgc{C}}\Ga'(\bar{y}),\All{x}\f(x)\seq\x.
 \]
For subcase (ii), suppose that $u\in\bar{z}$. By the induction hypothesis, there exist a sentence $\x'$ and derivations $d'_{1},d'_{2}$ such that $\md(d'_1),\md(d'_2)\le\md(d')$ and
 \[
 d'_1\der{\lgc{C}}\Ga'(\bar{y})\seq\x',\qquad d'_2\der{\lgc{C}}\Pi(\bar{z}),\f(u),\x'\seq\De(\bar{z}).
 \]
 We define $\x:=\x'\pd\All{x}\f(x)$ and obtain derivations $d_1,d_2$ satisfying $\md(d_1)=\md(d'_1)\le\md(d')<\md(d)$, $\md(d_2)=\md(d'_2)+1\le\md(d')+1=\md(d)$, and
 \begin{align*}
 d_1\der{\lgc{C}}\Ga'(\bar{y}),\All{x}\f(x)\seq\x'\pd\All{x}\f(x), \qquad 
  d_2\der{\lgc{C}}\Pi(\bar{z}),\x'\pd\All{x}\f(x)\seq\De(\bar{z}).
\end{align*}
Suppose next that $\Pi(\bar{z})$ is $\Pi'(\bar{z}),\All{x}\f(x)$ and
\[
d'\der{\lgc{C}}\Ga(\bar{y}),\Pi'(\bar{z}),\f(u)\seq\De(\bar{z}).
\]
 The case of $u\in\bar{z}$ is  very similar to subcase (i) above, so suppose that $u\in\bar{y}$. By the induction hypothesis, there exist a sentence $\x'$ and derivations $d'_{1},d'_{2}$ such that  $\md(d'_1),\md(d'_2)\le\md(d')$ and
 \[
 d'_1\der{\lgc{C}}\Ga'(\bar{y}),\f(u)\seq\x', \qquad
  d'_2\der{\lgc{C}}\Pi'(\bar{z}),\x'\seq\De(\bar{z}).
  \] 
  We let $\x:=\All{x}\f(x)\to\x'$ and obtain derivations $d_1,d_2$ satisfying  $\md(d_1)=\md(d'_1)+1\le\md(d')+1=\md(d)$, $\md(d_2)=\md(d'_2)<\md(d)$, and
 \begin{align*}
 d_1\der{\lgc{C}}\Ga(\bar{y})\seq\All{x}\f(x)\to\x',\quad
d_2\der{\lgc{C}}\Pi'(\bar{z}),\All{x}\f(x),\All{x}\f(x)\to\x'\seq\De(\bar{z}).
\end{align*}
\noindent $\bullet$ $\farr$: Suppose that $\De(\bar{z})$ is $\All{x}\f(x)$ and for some variable $u$ that does not occur freely in $\Ga(\bar{y}),\Pi(\bar{z})\seq\All{x}\f(x)$,
\[
d'\der{\lgc{C}}\Ga(\bar{y}),\Pi(\bar{z})\seq\f(u),
\]
where $\md(d')=\md(d)-1$. By the induction hypothesis, there exist a sentence $\x$ and derivations $d_{1},d'_{2}$ such that  $\md(d_1),\md(d'_2)\le\md(d')$ and
\[
d_1\der{\lgc{C}}\Ga(\bar{y})\seq\x,\qquad d'_2\der{\lgc{C}}\Pi(\bar{z}),\x\seq\f(u).
\]
An application of $\farr$ yields a derivation $d_2$ satisfying $\md(d_2)=\md(d'_2)+1\le\md(d')+1=\md(d)$ and
$
d_2\der{\lgc{C}}\Pi(\bar{z}),\x\seq\All{x}\f(x). 
$
\end{proof}

\noindent 
{\bf Alternative proof of Theorem~\ref{t:FLecompleteness}(b).}
The right-to-left direction follows from Corollary~\ref{c:soundness2}.  For the converse, let $\V$ be $\cls{FL}_e$, $\cls{FL}_{ew}$, or $\cls{FL}_{ec}$, and let $\lgc{C}$ be $\fFLe$, $\fFLew$, or $\fFLec$, respectively. We note first that due to compactness and the local deduction theorem for $\fosc{\overline{\V}}$ (see~\cite{CN21}*{Sections~4.6, 4.8}), we can restrict to the case where $T = \emptyset$. Hence, by  Proposition~\ref{p:ono}, it suffices to prove that for any sequent $\Ga\seq\De$ consisting only of formulas from $\ofml(\langs)$, 
\[
\textstyle
d\der{\lgc{C}}\Ga\seq\De
\quad\Longrightarrow\quad
\vDash_{\mV} (\prod\Ga)^\ast \le (\sum\De)^\ast.
\]
We proceed by induction on the lexicographically ordered pair $\tuple{\md(d),\height(d)}$, where $\height(d)$ is the height of the derivation $d$. The base cases are clear and all the cases for the last application of a rule in $d$ except $\farr$ and $\elr$ follow by applying the induction hypothesis and the equations defining $\mV$. Just note that for each such rule, the premises contain only formulas from $\ofml(\langs)$ with at least one fewer symbol. In particular, for $\falr$ and $\err$, the term $t$ occurring in the premise must be $x$ and the result follows using {\rm (L1$_\bo$)} or {\rm (L1$_\di$)}.

Consider now a last application of $\farr$ in $d$, where $\De$ is $\All{x}\p(x)$. Then $d'\der{\lgc{C}}\Ga\seq\p(y)$ for some variable $y$ that does not occur freely in $\Ga\seq\All{x}\p(x)$, where $\md(d')=\md(d)-1$. If $y=x$, then $x$ does not occur freely in $\Ga$ and the result follows by an application of the induction hypothesis and equations defining $\mV$. Suppose that $y\neq x$. By Theorem~\ref{t:interpolation}, there exist a sentence $\x$ and derivations $d_1,d_2$ such that $d_1\der{\lgc{C}}\Ga\seq\x$ and $d_2\der{\lgc{C}}\x\seq\p(y)$ with $\md(d_1),\md(d_2)\le\md(d')$. Since $\x$ is a sentence and $y\neq x$, we can substitute in $d_2$ every free occurrence of $x$ by some new variable $z$, and then every occurrence of $y$ by $x$, to obtain a derivation $d'_2$ of $\x\seq\p(x)$ with $\md(d'_2)=\md(d_2)$. By the induction hypothesis, $\vDash_{\mV} (\prod\Ga)^\ast \le \x^\ast$ and $\vDash_{\mV} \x^\ast \le\p(x)^\ast$. Since $\x$ is a sentence,  the equations defining $\mV$ yield also $\vDash_{\mV} \x^\ast \le (\All{x}\p(x))^\ast$. So  $\vDash_{\mV} \Ga^\ast \le (\All{x}\p(x))^\ast$.

Consider finally a last application of $\elr$ in $d$, where $\Ga$ is $\Ga',\Exi{x}\p(x)$. Then $d'\der{\lgc{C}}\Ga',\p(y)\seq\De$ for some variable $y$  that does not occur freely in $\Ga',\Exi{x}\p(x)\seq\De$, where $\md(d')<\md(d)$. If $y=x$, then $x$ does not occur freely in $\Ga'$ or $\De$ and the result follows by applying the induction hypothesis and equations defining $\mV$. Suppose that $y\neq x$. By Theorem~\ref{t:interpolation}, there exist a sentence $\x$ and derivations $d_1,d_2$ such that $d_1\der{\lgc{C}}\p(y)\seq\x$ and $d_2\der{\lgc{C}}\Ga',\x\seq\De$ with $\md(d_1),\md(d_2)\le\md(d')$. Since $\x$ is a sentence and $y\neq x$, we can substitute in $d_1$ every free occurrence of $x$ by some new variable $z$, and then every occurrence of $y$ by $x$, to obtain a  derivation $d'_1$ of $\p(x)\seq\x$ with $\md(d'_1)=\md(d')$. By the induction hypothesis, $\vDash_{\mV} \p(x)^\ast \le\x^\ast$ and $\vDash_{\mV} (\prod(\Ga',\x))^\ast \le(\sum\De)^\ast$. Since $\x$ is a sentence, the equations defining $\mV$ yield also $\vDash_{\mV} (\Exi{x}\p(x))^\ast \le\x^\ast$. So $\vDash_{\mV} (\prod(\Ga', \Exi{x}\p(x)))^\ast  \le(\sum\De)^\ast$. \qed


\section{Concluding Remarks}\label{s:concluding}

Let us conclude this paper by sketching a broader perspective. Given some class $\K$ of complete $\lang$-lattices, the challenge is to find a (natural) axiomatization of the generalized quasivariety generated by the class of full functional m-$\lang$-lattices of the form $\tuple{\alg{A}^W,\bo,\di}$ for some $\alg{A}\in\K$ and set $W$. Corollary~\ref{c:completeness} shows that in the case where $\K$ is the class of complete members of a variety $\V$ that admits regular completions and has the superamalgamation property, this generalized quasivariety is in fact the variety $\mV$. In general, however, we may need to add further axioms to obtain a proper subvariety or even proper subquasivariety or sub-generalized quasivariety of $\mV$.

For example, let $\V$ be a variety of {\em semilinear} \FLe-algebras, that is, algebras that are isomorphic to a subdirect product of totally ordered \FLe-algebras. In general, such varieties may not admit regular completions (e.g., as shown in \cite{GP02} for $\V = \cls{MV}$), so Corollary~\ref{c:completeness} may not apply.  Moreover, using the fact that $\V$ is generated by totally ordered \FLe-algebras, $\fosc{\overline{\V}}\Exi{x}\f\pd\Exi{x}\f \eq \Exi{x}(\f\pd \f)$, while, as proved in Example~\ref{e:monadicvarieties}, if $\textbf{\L}_3\in\V$ (e.g., if $\V$ is $\cls{MV}$ or the variety of all semilinear \FLe-algebras), then $\not\vDash_{\mV}\di x\pd\di x \eq \di(x\pd x)$. 

In fact, for a variety $\V$ of semilinear \FLe-algebras, first-order semantical consequence is typically defined with respect to the class of its complete totally ordered members. In this case, the corresponding generalized quasivariety will satisfy also the constant domain axiom $\bo(\bo x\jn y)\eq\bo x\jn\bo y$. Indeed, if $\V$ is $\cls{MV}$, this generalized quasivariety is the variety of monadic MV-algebras axiomatized as the subvariety of $\mcls{MV}$ satisfying $\di x\pd\di x \eq \di(x\pd x)$ and the constant domain axiom~\cites{Rut59}. Interestingly, a proof of this latter result is given in~\cites{CCVR20} using the fact that the class of totally ordered MV-algebras has the amalgamation property (see also~\cites{MT20,Tuy21} for related results), suggesting that the approach developed in this paper can be adapted to a broader class of one-variable lattice logics.


\bibliographystyle{aiml22}


\appendix

\section{Missing cases for the proof of Theorem~\ref{t:interpolation}}

\noindent $\bullet$ $\elr$: Suppose first that $\Ga(\bar{y})$ is $\Ga'(\bar{y}),(\exists x)\f(x)$ and for some variable $u$ that does not occur freely in $\Ga'(\bar{y}),(\exists x)\f(x),\Pi(\bar{z})\seq\De(\bar{z})$,
\[
d'\der{\lgc{C}}\Ga'(\bar{y}),\f(u),\Pi(\bar{z})\seq\De(\bar{z}),
\]
where $\md(d')=\md(d)-1$. Let $\bar{y'}:=\bar{y}\cup\{u\}$ and $\hat{\Ga}(\bar{y'}):=\Ga'(\bar{y})\cup\{\f(u)\}$.  By the induction hypothesis, we obtain a sentence $\x$ and derivations $d_1',d_2$ such that $\md(d_1'),\md(d_2)\leq\md(d')=\md(d)-1$ and
\[
d_1'\der{\lgc{C}}\hat{\Ga}(\bar{y'})\seq\x,\qquad d_2\der{\lgc{C}}\Pi(\bar{z}),\x\seq\De(\bar{z}).
\]
The derivation $d_1'$ together with an application of $\elr$ yields a derivation $d_1$ such that $\md(d_1)=\md(d_1')+1\leq\md(d')+1=\md(d)$ and 
\[
d_1\der{\lgc{C}}\Ga'(\bar{y}),(\exists x)\f(x)\seq\x.
\]
Now suppose that $\Pi(\bar{z})$ is $\Pi'(\bar{z}),(\exists x)\f(x)$ and for some variable $u$ that does not occur freely in $\Ga(\bar{y}),\Pi'(\bar{z}),(\exists x)\f(x)\seq\De(\bar{z})$,
\[
d'\der{\lgc{C}}\Ga(\bar{y}),\Pi'(\bar{z}),\f(u)\seq\De(\bar{z}),
\]
where $\md(d')=\md(d)-1$. We let $\bar{z'}:=\bar{z}\cup\{u\}$ and $\hat{\Pi}(\bar{z'}):=\Pi'(\bar{z})\cup\{\f(u)\}$. Note that  $\De(\bar{z'})=\De(\bar{z})$. By the induction hypothesis, we obtain a sentence $\x$ and derivations $d_1,d_2'$ such that $\md(d_1),\md(d_2')\leq\md(d')=\md(d)-1$ and
\[
d_1\der{\lgc{C}}\Ga(\bar{y})\seq\x,\qquad d_2'\der{\lgc{C}}\hat{\Pi}(\bar{z'}),\x\seq\De(\bar{z'}).
\]
The derivation $d_2'$ together with an application of $\elr$ yields a derivation $d_2$ such that $\md(d_2)=\md(d_2')+1\leq\md(d')+1=\md(d)$ and 
\[
d_2\der{\lgc{C}}\Pi'(\bar{z}),(\exists x)\f(x),\x\seq\De(\bar{z}).
\]

\medskip\noindent
$\bullet$ $\err$: Suppose that $\De(\bar{z})$ is $(\exists x)\f(x)$ and there is a derivation $d'$ such that $\md(d')=\md(d)-1$ and
\[
d'\der{\lgc{C}}\Ga(\bar{y}),\Pi(\bar{z})\seq\f(u).
\]
There are two subcases. For subcase (i), suppose that $u\in\bar{y}$. By the induction hypothesis, there exist a sentence $\x'$ and derivations $d_1',d_2'$ such that $\md(d_1'),\md(d_2')\leq\md(d')=\md(d)-1$ and
\[
d_1'\der{\lgc{C}}\Pi(\bar{z})\seq\x',\qquad d_2'\der{\lgc{C}}\Ga(\bar{y}),\x'\seq\f(u).
\]
Let $\x:=\x'\to(\exists x)\f(x)$. Then $d_1'$, together with the derivation $\hat{d}\der{\lgc{C}}(\exists x)\f(x)\seq(\exists x)\f(x)$ and an application of $\ilr$, yields a derivation $d_2$, and $d_2'$, together with applications of  $\err$ and $\irr$, yields a derivation $d_1$, satisfying $\md(d_2)=\md(d_1')$, $\md(d_1)=\md(d_2')+1\leq\md(d')+1=\md(d)$ and
\[
d_2\der{\lgc{C}}\Pi(\bar{z}),\x'\to(\exists x)\f(x)\seq(\exists x)\f(x),\qquad d_1\der{\lgc{C}}\Ga(\bar{y})\seq\x'\to(\exists x)\f(x).
\]

For subcase (ii), suppose that $u\in\bar{z}$. By the induction hypothesis, there exist a sentence $\x$ and derivations $d_1,d_2'$ such that $\md(d_1),\md(d_2')\leq\md(d')=\md(d)-1$ and
\[
d_1\der{\lgc{C}}\Ga(\bar{y})\seq\x,\qquad d_2'\der{\lgc{C}}\Pi(\bar{z}),\x\seq\f(u).
\]
The derivation $d_2'$ with an application of $\err$ yields a derivation $d_2$ such that $\md(d_2)=\md(d_2')+1\leq\md(d')+1=\md(d)$ and
\[
d_2\der{\lgc{C}}\Pi(\bar{z}),\x\seq(\exists x)\f(x).
\]

\medskip\noindent
$\bullet$ $\ilr$: Suppose first that $\Ga(\bar{y})$ is $\Ga_1(\bar{y}),\Ga_2(\bar{y}),\f(\bar{y})\to\p(\bar{y})$ and $\Pi(\bar{z})$ is $\Pi_1(\bar{z}),\Pi_2(\bar{z})$, and 
\[
d'_1\der{\lgc{C}}\Ga_1(\bar{y}),\Pi_1(\bar{z})\seq\f(\bar{y}), \qquad d'_2\der{\lgc{C}}\Ga_2(\bar{y}),\p(\bar{y}),\Pi_2(\bar{z})\seq\De(\bar{z}).
\]
By the induction hypothesis, there exist sentences $\x_1,\x_2$ and derivations $d'_{11},d'_{12},d'_{21},d'_{22}$ such that 
\begin{align*}
d'_{11}\der{\lgc{C}}\Ga_1(\bar{y}),\x_1\seq\f(\bar{y}), &  \qquad d'_{12}\der{\lgc{C}}\Pi_1(\bar{z})\seq\x_1,\\
d'_{21}\der{\lgc{C}}\Ga_2(\bar{y}),\p(\bar{y})\seq\x_2, & \qquad d'_{22}\der{\lgc{C}}\Pi_2(\bar{z}),\x_2\seq\De(\bar{z}).
\end{align*} 
 Let  $\x:=\x_1\to\x_2$. Then the derivations $d'_{11},d'_{21}$, together with applications of $\ilr$ and $\irr$, yield a derivation $d_1$, and the derivations $d'_{12},d'_{22}$, together with an application of $\ilr$ yield a derivation $d_2$ satisfying
  \begin{align*}
 & d_{1}\der{\lgc{C}}\Ga_1(\bar{y}),\Ga_2(\bar{y}),\f(\bar{y})\to\p(\bar{y})\seq\x_1\to\x_2,\\
& d_{2}\der{\lgc{C}}\Pi_1(\bar{z}),\Pi_2(\bar{z}),\x_1\to\x_2\seq\De(\bar{z}).
 \end{align*}
Clearly, the constraints on $\md (d_1)$ and $\md(d_2)$ are satisfied.

Now suppose that $\Ga(\bar{y})$ is $\Ga_1(\bar{y}),\Ga_2(\bar{y})$ and $\Pi(\bar{z})$ is $\Pi_1(\bar{z}),\Pi_2(\bar{z}),\f(\bar{z})\to\p(\bar{z})$, and 
\[
d'_1\der{\lgc{C}}\Ga_1(\bar{y}),\Pi_1(\bar{z})\seq\f(\bar{z}), \qquad
d'_2\der{\lgc{C}}\Ga_2(\bar{y}),\p(\bar{z}),\Pi_2(\bar{z})\seq\De(\bar{z}). 
\]
By the induction hypothesis, there exist sentences $\x_1,\x_2$ and derivations $d'_{11},d'_{12},d'_{21},d'_{22}$ such that 
\begin{align*}
d'_{11}\der{\lgc{C}}\Ga_1(\bar{y})\seq\x_1, & \qquad d'_{12}\der{\lgc{C}}\Pi_1(\bar{z}),\x_1\seq\f(\bar{z}),\\
d'_{21}\der{\lgc{C}}\Ga_2(\bar{y})\seq\x_2, & \qquad d'_{22}\der{\lgc{C}}\Pi_2(\bar{z}),\p(\bar{z}),\x_2\seq\De(\bar{z}).
\end{align*} 
Let  $\x:=\x_1\pd\x_2$. Then the derivations $d'_{11},d'_{21}$, together with an application of $\pdrr$, and the derivations $d'_{12},d'_{22}$, together with applications of $\ilr$ and $\pdlr$, yield derivations $d_1$ and $d_2$, respectively,  such that
\begin{align*}
 & d_{1}\der{\lgc{C}}\Ga_1(\bar{y}),\Ga_2(\bar{y})\seq\x_1\pd\x_2,\\
& d_{2}\der{\lgc{C}}\Pi_1(\bar{z}),\Pi_2(\bar{z}),\f(\bar{z})\to\p(\bar{z}),\x_1\pd\x_2\seq\De(\bar{z}).
 \end{align*}
Again, the constraints on $\md(d_1)$ and $\md(d_2)$ are clearly satisfied in this case.

\pagebreak

\medskip\noindent
$\bullet$ $\irr$: Suppose that $\De(\bar{z})$ is $\f(\bar{z})\to\p(\bar{z})$ and
\[
d'\der{\lgc{C}}\Ga(\bar{y}),\Pi(\bar{z}),\f(\bar{z})\seq\p(\bar{z}).
\]
By the induction hypothesis, there exist a sentence $\x$ and derivations $d_1,d_2'$ such that
\[
d_1\der{\lgc{C}}\Ga(\bar{y})\seq\x,\qquad d_2'\der{\lgc{C}}\Pi(\bar{z}),\f(\bar{z}),\x\seq\p(\bar{z}).
\]
The derivation $d_2'$ with an application of $\irr$ yields a derivation $d_2$ such that
\[
d_2\der{\lgc{C}}\Pi(\bar{z}),\x\seq\f(\bar{z})\to\p(\bar{z}).
\]
The constraints on $\md(d_1)$ and $\md(d_2)$ clearly hold.

\medskip\noindent
$\bullet$ $\olr$: Suppose first that $\Ga(\bar{y})$ is $\Ga'(\bar{y}),\f_1(\bar{y})\jn\f_2(\bar{y})$ and
\[
d_1'\der{\lgc{C}}\Ga'(\bar{y}),\f_1(\bar{y}),\Pi(\bar{z})\seq\De(\bar{z}),\qquad d_2'\der{\lgc{C}}\Ga'(\bar{y}),\f_2(\bar{y}),\Pi(\bar{z})\seq\De(\bar{z}).
\]
By the induction hypothesis, there exist sentences $\x_1,\x_2$ and derivations $d'_{11},d'_{12},d'_{21},d'_{22}$ such that
\begin{align*}
&d'_{11}\der{\lgc{C}}\Ga'(\bar{y}),\f_1(\bar{y})\seq\x_1,\qquad d'_{12}\der{\lgc{C}}\Pi(\bar{z}),\x_1\seq\De(\bar{z}),\\
&d'_{21}\der{\lgc{C}}\Ga'(\bar{y}),\f_2(\bar{y})\seq\x_2,\qquad d'_{22}\der{\lgc{C}}\Pi(\bar{z}),\x_2\seq\De(\bar{z}).
\end{align*}
Define $\x:=\x_1\jn\x_2$. The derivations $d'_{11},d'_{21}$, together with applications of $\orr_1$, $\orr_2$, and $\olr$, yield a derivation $d_1$, and  the derivations $d'_{12},d'_{22}$, together with an application of $\olr$, yield a derivation $d_2$, satisfying
\[
d_1\der{\lgc{C}}\Ga'(\bar{y}),\f_1(\bar{y})\jn\f_2(\bar{y})\seq\x_1\jn\x_2,\qquad d_2\der{\lgc{C}}\Pi(\bar{z}),\x_1\jn\x_2\seq\De(\bar{z}).
\]
The constraints on $\md(d_1)$ and $\md(d_2)$ clearly hold. 

Suppose now that $\Pi(\bar{z})$ is $\Pi'(\bar{z}),\f_1(\bar{z})\jn\f_2(\bar{z})$ and
\[
d_1'\der{\lgc{C}}\Ga(\bar{y}),\Pi'(\bar{z}),\f_1(\bar{z})\seq\De(\bar{z}),\qquad d_2'\der{\lgc{C}}\Ga(\bar{y}),\Pi'(\bar{z}),\f_2(\bar{z})\seq\De(\bar{z}).
\]
By the induction hypothesis, there exist sentences $\x_1,\x_2$ and derivations $d'_{11},d'_{12},d'_{21},d'_{22}$ such that
\begin{align*}
&d'_{11}\der{\lgc{C}}\Ga(\bar{y})\seq\x_1,\qquad d'_{12}\der{\lgc{C}}\Pi'(\bar{z}),\f_1(\bar{z}),\x_1\seq\De(\bar{z}),\\
&d'_{21}\der{\lgc{C}}\Ga(\bar{y})\seq\x_2,\qquad d'_{22}\der{\lgc{C}}\Pi'(\bar{z}),\f_2(\bar{z}),\x_2\seq\De(\bar{z}).
\end{align*}
Let $\x:=\x_1\mt\x_2$. Then the derivations $d'_{11},d'_{21}$, together with an application of $\arr$, and the derivations $d'_{12},d'_{22}$, together with applications of $\alr_1$, $\alr_2$, and  $\olr$, yield derivations $d_1$ and $d_2$, respectively, such that
\[
d_1\der{\lgc{C}}\Ga(\bar{y})\seq\x_1\mt\x_2,\qquad d_2\der{\lgc{C}}\Pi'(\bar{z}),\f_1(\bar{z})\jn\f_2(\bar{z}),\x_1\mt \x_2\seq\De(\bar{z}).
\]
The constraints on $\md(d_1)$ and $\md(d_2)$ again clearly hold.

\medskip\noindent
$\bullet$ $\orr_i$ ($i\in\{1,2\}$):  Suppose that $\De(\bar{z})$ is $\f_1(\bar{z})\jn\f_2(\bar{z})$ and
\[
d'\der{\lgc{C}}\Ga(\bar{y}),\Pi(\bar{z})\seq\f_i(\bar{z}).
\]
By the induction hypothesis, there exist a sentence $\x$ and derivations $d_1,d_2'$ such that
\[
d_1\der{\lgc{C}}\Ga(\bar{y})\seq\x,\qquad d_2'\der{\lgc{C}}\Pi(\bar{z}),\x\seq\f_i(\bar{z}).
\]
The derivation $d_2'$ together with an application of $\orr_i$ yields a derivation $d_2$ such that
\[
d_2\der{\lgc{C}}\Pi(\bar{z}),\x\seq\f_1(\bar{z})\jn\f_2(\bar{z}).
\]
The constraints on $\md(d_1)$ and $\md(d_2)$ clearly hold.

\medskip\noindent
$\bullet$ $\arr$: Suppose that $\De(\bar{z})$ is $\p_1(\bar{z})\mt\p_2(\bar{z})$ and
\[
d'_1\der{\lgc{C}}\Ga(\bar{y}),\Pi(\bar{z})\seq\p_1(\bar{z}),\qquad d'_2\der{\lgc{C}}\Ga(\bar{y}),\Pi(\bar{z})\seq\p_2(\bar{z}).
\]
By the induction hypothesis, there exist sentences $\x_1,\x_2$ and derivations $d'_{11},d'_{12},d'_{21},d'_{22}$ such that 
\begin{align*}
d'_{11}\der{\lgc{C}}\Ga(\bar{y})\seq\x_1, & \qquad d'_{12}\der{\lgc{C}}\Pi(\bar{z}),\x_1\seq\p_1(\bar{z}),\\
d'_{21}\der{\lgc{C}}\Ga(\bar{y})\seq\x_2, & \qquad d'_{22}\der{\lgc{C}}\Pi(\bar{z}),\x_2\seq\p_2(\bar{z}).
\end{align*} 
Let $\x:=\x_1\mt\x_2$. Then the derivations $d'_{11},d'_{21}$, together with an application of $\arr$, and the derivations $d'_{12},d'_{22}$, together with applications of $\alr_1$, $\alr_2$, and $\arr$, yield derivations $d_1$ and $d_2$, respectively, such that
\[
d_1\der{\lgc{C}}\Ga(\bar{y})\seq\x_1\mt\x_2,\qquad d_2\der{\lgc{C}}\Pi(\bar{z}),\x_1\mt\x_2\seq\p_1(\bar{z})\mt\p_2(\bar{z}).
\]
Clearly, the constraints on $\md(d_1)$ and $\md(d_2)$ are satisfied in this case.

\medskip\noindent
$\bullet$ $\alr_i$ ($i\in\{1,2\}$): Suppose first that $\Ga(\bar{y})$ is $\Ga'(\bar{y}),\f_1(\bar{y})\mt\f_2(\bar{y})$ and
\[
d'\der{\lgc{C}}\Ga'(\bar{y}),\f_i(\bar{y}),\Pi(\bar{z})\seq\De(\bar{z}).
\]
By the induction hypothesis, there exist a sentence $\x$ and derivations $d_1',d_2$ such that
\[
d_1'\der{\lgc{C}}\Ga'(\bar{y}),\f_i(\bar{y})\seq\x,\qquad d_2\der{\lgc{C}}\Pi(\bar{z}),\x\seq\De(\bar{z}).
\]
The derivation $d_1'$ and an application of $\alr_i$ yield a derivation $d_1$ satisfying
\[
d_1\der{\lgc{C}}\Ga'(\bar{y}),\f_1(\bar{y})\mt\f_2(\bar{y})\seq\x.
\]
The constraints on $\md(d_1)$ and $\md(d_2)$ clearly hold.

Now suppose that $\Pi(\bar{z})$ is $\Pi'(\bar{z}),\f_1(\bar{z})\mt\f_2(\bar{z})$ and
\[
d'\der{\lgc{C}}\Ga(\bar{y}),\Pi'(\bar{z}),\f_i(\bar{z})\seq\De(\bar{z}).
\]
By the induction hypothesis, there exist a sentence $\x$ and derivations $d_1,d_2'$ such that
\[
d_1\der{\lgc{C}}\Ga(\bar{y})\seq\x,\qquad d_2'\der{\lgc{C}}\Pi'(\bar{z}),\f_i(\bar{z}),\x\seq\De(\bar{z}).
\]
The derivation $d_2'$ together with an application of $\alr_i$ yields a derivation $d_2$ such that
\[
d_2\der{\lgc{C}}\Pi'(\bar{z}),\f_1(\bar{z})\mt\f_2(\bar{z}),\x\seq\De(\bar{z}).
\]
Again, the constraints on $\md (d_1)$ and $\md(d_2)$ hold. 

\pagebreak

\medskip\noindent
$\bullet$ $\pdrr$: Suppose that $\De(\bar{z})$ is $\f(\bar{z})\pd\p(\bar{z})$, $\Ga(\bar{y})$ is $\Ga_1(\bar{y}),\Ga_2(\bar{z})$, and $\Pi(\bar{z})$ is $\Pi_1(\bar{z}),\Pi_2(\bar{z})$, and
\[
d_1'\der{\lgc{C}}\Ga_1(\bar{y}),\Pi_1(\bar{z})\seq\f(\bar{z}),\qquad d_2'\der{\lgc{C}}\Ga_2(\bar{y}),\Pi_2(\bar{z})\seq\p(\bar{z}).
\]
By the induction hypothesis, there exist sentences $\x_1,\x_2$ and derivations $d'_{11},d'_{12},d'_{21},d'_{22}$ such that
\begin{align*}
d'_{11}\der{\lgc{C}}\Ga_1(\bar{y})\seq\x_1, & \qquad d'_{12}\der{\lgc{C}}\Pi_1(\bar{z}),\x_1\seq\f(\bar{z}),\\
d'_{21}\der{\lgc{C}}\Ga_2(\bar{y})\seq\x_2, & \qquad d'_{22}\der{\lgc{C}}\Pi_2(\bar{z}),\x_2\seq\p(\bar{z}).
\end{align*}
Let $\x:=\x_1\pd\x_2$. Then the derivations $d'_{11},d'_{21}$, together with an application of $\pdrr$, and the derivations $d'_{12},d'_{22}$, together with applications of $\pdrr$ and $\pdlr$, yield derivations $d_1$ and $d_2$, respectively, such that
\begin{align*}
&d_1\der{\lgc{C}}\Ga_1(\bar{y}),\Ga_2(\bar{y})\seq\x_1\pd\x_2,\qquad
d_2\der{\lgc{C}}\Pi_1(\bar{z}),\Pi_2(\bar{z}),\x_1\pd\x_2\seq\f(\bar{z})\pd\p(\bar{z}).
\end{align*}
The constraints on $\md(d_1)$ and $\md(d_2)$ clearly hold. 

\medskip\noindent
$\bullet$ $\pdlr$: Suppose first that $\Ga(\bar{y})$ is $\Ga'(\bar{y}),\f(\bar{y})\pd\p(\bar{y})$ and
\[
d'\der{\lgc{C}}\Ga'(\bar{y}),\f(\bar{y}),\p(\bar{y}),\Pi(\bar{z})\seq\De(\bar{z}).
\]
By the induction hypothesis, there exist a sentence $\x$ and derivations $d_1',d_2$ such that
\[
d_1'\der{\lgc{C}}\Ga'(\bar{y}),\f(\bar{y}),\p(\bar{y})\seq\x,\qquad d_2\der{\lgc{C}}\Pi(\bar{z}),\x\seq\De(\bar{z}).
\]
Then $d_1'$ and an application of $\pdlr$ yield a derivation $d_1$ such that
\[
d_1\der{\lgc{C}}\Ga'(\bar{y}),\f(\bar{y})\pd\p(\bar{y})\seq\x.
\]
The constraints on $\md(d_1)$ and $\md(d_2)$ clearly hold. 

Now suppose that $\Pi(\bar{z})$ is $\Pi'(\bar{z}),\f(\bar{z})\pd\p(\bar{z})$ and
\[
d'\der{\lgc{C}}\Ga(\bar{y}),\Pi'(\bar{z}),\f(\bar{z}),\p(\bar{z})\seq\De(\bar{z}).
\]
By the induction hypothesis, there exist a sentence $\x$ and derivations $d_1,d_2'$ such that
$$d_1\der{\lgc{C}}\Ga(\bar{y})\seq\x\quad\text{and}\quad d_2'\der{\lgc{C}}\Pi'(\bar{z}),\f(\bar{z}),\p(\bar{z}),\x\seq\De(\bar{z}).$$
Taking $d_2'$ and applying $\pdlr$ then yields a derivation $d_2$ such that
$$d_2\der{\lgc{C}}\Pi'(\bar{z}),\f(\bar{z})\pd\p(\bar{z}),\x\seq\De(\bar{z}).$$
Again, $\md(d_1)$ and $\md(d_2)$ satisfy the constraints.

\end{document}